\documentclass[12pt,reqno]{amsart}
\usepackage{amsmath}
\usepackage{amssymb}
\usepackage{amsthm}
\usepackage{enumerate}
\usepackage[mathscr]{eucal}
\usepackage{eqlist}
\usepackage{hyperref}
\usepackage{cite}
\usepackage[polish,english]{babel}
\usepackage{polski}
\hypersetup{colorlinks=true,linkcolor=red, anchorcolor=green, citecolor=blue, urlcolor=red, filecolor=magenta, pdftoolbar=true}

\newtheorem{thm}         {Theorem}
\newtheorem{cor}   [thm] {Corollary}

\theoremstyle{definition}

\newtheorem{rem}   [thm] {Remark}
\newtheorem{qes}   [thm] {Questions}

    {\large\bfseries Acknowledgement%
    \par\medskip\normalfont\normalsize}%
    {}%
\usepackage{xcolor}

\numberwithin{equation}{section}

\begin{document}
\setcounter{page}{1}

\title[A new type of   functional equations]
{A new type of   functional equations  on semigroups with  involutions
}
\author[Iz. EL-Fassi]%
{Iz-iddine EL-Fassi$^{*1}$  }

\newcommand{\acr}{\newline\indent}

\address{$^{1}$Department of Mathematics,\acr
                   Faculty of Sciences and Techniques,\acr
                   Sidi Mohamed Ben Abdellah University, B.P. 2202 \acr
                   Fez,
                  Morocco}

					\address{\vspace*{5mm}		}									
\email{ $^{1}$izidd-math@hotmail.fr; iziddine.elfassi@usmba@ac.ma;\newline\hspace*{3cm} izelfassi.math@gmail.com}

\renewcommand{\thefootnote}{}
\footnotetext{ $^{*}$Corresponding author.}

\begin{abstract} Let $S$ be a commutative semigroup, $K$  a quadratically closed commutative field of characteristic different from $2$, $G$  a $2$-cancellative abelian group and $H$  an abelian group uniquely divisible by $2$. The aim of this paper is to determine the general solution $f:S^2\to K$ of the d'Alembert type equation:
$$
f(x+y,z+w)+f(x+\sigma(y),z+\tau(w)) =2f(x,z)f(y,w),\quad\quad (x,y,z,w\in S)
$$
the general solution $f:S^2\to G$ of the Jensen type equation:
$$
f(x+y,z+w)+f(x+\sigma(y),z+\tau(w)) =2f(x,z),\quad\quad (x,y,z,w\in S)
$$
the general solution $f:S^2\to H$ of the quadratic type equation quation:
$$
f(x+y,z+w)+f(x+\sigma(y),z+\tau(w)) =2f(x,z)+2f(y,w),\quad\quad (x,y,z,w\in S)
$$
where  $\sigma,\tau: S\to S$ are two involutions. 
\end{abstract}

\maketitle

\section{Notation and terminology }
Throughout the paper we work in the following framework and with the following notation and terminology.

 $(S,+)$ is a commutative semigroup, $K$ is a field of characteristic different from $2$,  $G$  is a $2$-cancellative abelian group and $H$  is an abelian group uniquely divisible by $2$. 
\begin{itemize}
\item [(i)] A mapping $\sigma:S\to S$ is called involution if 
 $$\sigma(\sigma(x)) = x\;\;\text{and}\;\;\sigma(x+y)=\sigma(x)+\sigma(y),\;\;x,y\in S.$$
	\item [(ii)]We say that $\chi:S^2\to K$ is a  multiplication function if
	$$\chi(x+y,z+w)=\chi(x,z)\chi(y,w),\;\;\;x,y,z,w\in S.$$
	\item [(iii)] A function  $A:S^2\to G$ is called  additive  if
	$$A(x+y,z+w)=A(x,z)+A(y,w),\;\;\;x,y,z,w\in S.$$
	\item [(iv)] We say that  $B:S^2\times S^2\to H$ is a  biadditive function if
	$$B(u+v,w)=B(u,w)+B(v,w),\;\;u,v,w\in S^2,$$
	and
	$$B(u,v+w)=B(u,v)+B(u,w),\;\;\;u,v,w\in S^2.$$
\end{itemize}
\section{Introduction}
The functional equation
\begin{align} \label{almbr}
f(x+y)+f(x-y)=2f(x)f(y),\;\;x,y\in\mathbb{R}
\end{align}
is known as the scalar d'Alembert functional equation. It has a long history going back to d'Alembert \cite{alm}.
This functional equation has an obvious extension to groups, and even to semigroups with an involution, and a satisfactory theory of its solutions exists as described in  \cite[Chapter 9]{stet} (also, see  \cite{acz,dav1,dav2} for details and references). 

In 1989, Acz\'el and Dhombres \cite{acz} proved  that a mapping $Q : X\to  Y$ satisfies the quadratic functional equation
\begin{equation}\label{quadr}
Q(x+y)+Q(x-y)=2Q(x)+2Q(y)
\end{equation}
if and only if there exists a symmetric bi-additive mapping $b : X^2\to Y$ such that $Q(x)=b(x,x)$, where
$$b(x,y):=\frac{1}{4}[Q(x+y)+Q(x-y)],\;\; x,y\in X,$$
 where $X$ and $Y$ are two vector spaces. Later, many different quadratic functional equations were solved by numerous authors; for example, see \cite{Czerwik2002,Kannappan2009,Sahoo2011}.

 Let $X$ and $Y$ be real vector spaces. For a given involution $\sigma: X \to X,$ the functional equation
\begin{equation}\label{quadreqn+}
g(x+y)+g(x+\sigma(y))=2g(x)+2g(y),
\end{equation}
is called the quadratic functional equation with involution. According to \cite[Corollary 8]{stt}, a function $g : X \to Y$ is a solution of \eqref{quadreqn+} if and only if there exist an additive function $a : X \to Y,$ and a bi-additive symmetric function $b : X^2 \to Y$ such that $a(\sigma(x)) = a(x)$, $b(\sigma(x), y) =-b(x, y)$ and $g (x) = a(x) + b(x, x)$ for all $x, y \in X.$

In 2010, Sinopoulos \cite{sino}  determined the general solution of the following functional equations: 
\begin{align} \label{almbr+}
g(x+y)+g(x+\sigma(y))=2g(x)g(y),\;\;x,y\in S,
\end{align}
\begin{align} \label{almbr+1}
g(x+y)+g(x+\sigma(y))=2g(x),\;\;x,y\in S,
\end{align}
\begin{align} \label{almbr+2}
g(x+y)+g(x+\sigma(y))=2g(x)+2g(y),\;\;x,y\in S,
\end{align}
where $(S,+)$ is a commutative semi-group and $\sigma:S\to S$ is an involution.

The  equation \eqref{almbr+1}, in the case where $S$ is an abelian group
divisible by $2$ and $\sigma(x)= -x$, is equivalent to the Jensen equation
\begin{align} \label{jn}
J\left(\frac{x+y}{2}\right)=\frac{J(x)+J(y)}{2},
\end{align}
and is easily reduced to the  Cauchy equation (see \cite{ac0}).

The equation \eqref{almbr+2}, again with $\sigma(x)= -x$, plays a fundamental role in the
characterization of inner product spaces \cite{acz0,ac0,amir,jor}.

The aim of this paper is to solve the following functional equations:
\begin{align} \label{bialm}
f(x+y,z+w)+f(x+\sigma(y),z+\tau(w))=2f(x,z)f(y,w)
\end{align}
\begin{align} \label{bijns}
f(x+y,z+w)+f(x+\sigma(y),z+\tau(w))=2f(x,z)
\end{align}
\begin{align} \label{biquad}
f(x+y,z+w)+f(x+\sigma(y),z+\tau(w))=2f(x,z)+2f(y,w)
\end{align}
for all $x,y,z,w\in S,$ where $(S,+)$ is a commutative semi-group and $\sigma,\;\tau$ are two involutions. 

The functional equation
\begin{align} \label{bimul}
f(x+y,z+w)= f(x,z)f(y,w),\;\;x,y,z,w\in S
\end{align}
corresponds to $\sigma(x)=\tau(x)=x$ in \eqref{bialm}, and the functional equation
\begin{align} \label{bibb}
f(x+y,z+w)= f(x,z)+f(y,w),\;\;x,y,z,w\in S
\end{align}
corresponds to $\sigma(x)=\tau(x)=x$ in \eqref{biquad}. In 2007, Bae and Park \cite{park2005} introduced the functional equation:
\begin{align} \label{biqua}
f(x+y,z+w)+f(x-y,z-w)=2 f(x,z)+2 f(y,w).
\end{align}
The functional equation \eqref{biqua} corresponds to $\sigma(x)=\tau(x)=-x$ in \eqref{biquad}, where $S$ is an Abelian group. When $X=Y=\mathbb{R},$ the function $f:\mathbb{R}^2  \rightarrow \mathbb{R}$ given by $f(x,y)=a x^2+bxy+c y^2$ is a solution of \eqref{biqua} where $a,b$ and $c$ are fixed real numbers.   It is clear that when $(z,w)=(x,y)$ and $\sigma=\tau$ in \eqref{bialm}, \eqref{bijns} and \eqref{biquad}, we get the  functional equations \eqref{almbr+}, \eqref{almbr+1} and \eqref{almbr+2}, where $g(x):=f(x,x)$ for all $x\in S$.  In addition, if $S$ is an Abelian group, $\sigma(x)=-x$ and $z=w=0$ in \eqref{bialm}, \eqref{bijns} and \eqref{biquad}, we find the functional equations 
\eqref{almbr}, \eqref{jn} and \eqref{quadr} with $f(x):=f(x,0)$ for all $x\in S.$

\section{General solution of Eq. \eqref{bialm}}\label{SVresults}

\begin{thm}\label{t1}
Let $(S,+)$ be a commutative semigroup,  $K$  a field of
characteristic different from $2$, and let $\sigma,\tau: S\to S$ be involutions. Then, the general solution $f : S^2 \to K$ of the  d'Alembert type equation:
\begin{align}\label{dalbtype}
f(x+y,z+w)+f(x+\sigma(y),z+\tau(w)) =2f(x,z)f(y,w),\quad x,y,z,w\in S,
\end{align}
is
\begin{align}\label{dalbtype1}
f(x,y)=\frac{\chi(x,y)+\chi(\sigma(x),\tau(y))}{2},\quad x,y\in S,
\end{align}
where $\chi : S^2 \to K$ is an arbitrary multiplicative function (i.e., $\chi(x+y,z+w)=\chi(x,z)\chi(y,w)$ for $x,y,z,w\in S$).
\end{thm}

\begin{proof} Replacing $(y,w)$ by $(\sigma(y),\tau(w))$ in \eqref{dalbtype}, we get
\begin{align}\label{dalbtype2}
f(x+y,z+w)+f(x+\sigma(y),z+\tau(w)) =2f(x,z)f(\sigma(y),\tau(w))
\end{align}
for all $x,y,z,w\in S$. From \eqref{dalbtype} and \eqref{dalbtype2}, we obtain
\begin{align}\label{dalbtype3}
f(\sigma(y),\tau(w))=f(y,w), \;\;\;y,w\in S,
\end{align}
and hence
\begin{align}\label{dalbtype4}
f(\sigma(x)+y,\tau(z)+w)=f(x+\sigma(y),z+\tau(w)),\;\;\;x,y,z,w\in S.
\end{align}
We have two cases: $f(x+\sigma(y),z+\tau(w))=f(x+y,z+w)$ and $f(x+\sigma(y),z+\tau(w))\neq f(x+y,z+w).$

\textbf{Case 1:} If $f(x+\sigma(y),z+\tau(w))=f(x+y,z+w),$ then by \eqref{dalbtype}, we get 
\begin{align}\label{dalbtype5}
f(x+y,z+w)=f(x,z)f(y,w),\;\;\;x,y,z,w\in S,
\end{align}
i.e., 
$$f(u+v)=f(u)f(v)$$
with $u=(x,z),$ $v=(y,w)$ and $u+v=(x+y,z+w).$ So, $f$ is a multiplicative function and can be written in the form \eqref{dalbtype1}: $f(x,y)=\frac{1}{2}(f(x,y)+f(\sigma(x),\tau(y))$.

\textbf{Case 2:} If $f(x+\sigma(y),z+\tau(w))\neq f(x+y,z+w),$ then there exist $x_0,y_0,z_0,w_0\in S$ such that
\begin{align}\label{dalbtype6}
f(x_0+y_0,z_0+w_0)-f(x_0+\sigma(y_0),z_0+\tau(w_0)) \neq 0.
\end{align}
In this case, we define the function $F:S^2\to K$ by 
$$F(x,z)=f(x+y_0,z+w_0)-f(x+\sigma(y_0),z+\tau(w_0)),\;\;x,z\in S.$$
By \eqref{dalbtype6}, we have $F(x_0,z_0)\neq 0$ and from \eqref{dalbtype3} and \eqref{dalbtype4}, we get
\begin{align}\label{dalbtype7}
F(\sigma(x),\tau(z))=-F(x,z)\;\;\text{and}\;\;F(x+\sigma(y),z+\tau(w))=-F(\sigma(x)+y,\tau(z)+w)
\end{align}
for all $x,y,z,w\in S$. Also
$$F(x+y,z+w)=f(x+y+y_0,z+w+w_0)-f(x+y+\sigma(y_0),z+w+\tau(w_0))$$
\begin{align*}
F(x+\sigma(y),z+\tau(w))=f(x+\sigma(y)&+y_0,z+\tau(w)+w_0)\\&-f(x+\sigma(y)+\sigma(y_0),z+\tau(w)+\tau(w_0))
\end{align*}
for all $x,y,z,w\in S$. Adding these equations and using \eqref{dalbtype3}, we obtain
\begin{align*}
F(x+y,z+w)+&F(x+\sigma(y),z+\tau(w))\\&=2f(y,w)[f(x+y_0,z+w_0)-f(x+\sigma(y_0),z+\tau(w_0))]
\end{align*}
for all $x,y,z,w\in S$. Hence
\begin{align}\label{dalbtype8}
F(x+y,z+w)+F(x+\sigma(y),z+\tau(w))=2f(y,w)F(x,z),\;\;x,y,z,w\in S.
\end{align}
Interchanging $x,y$ and $z,w$ in \eqref{dalbtype8}, we get
$$F(x+y,z+w)+F(y+\sigma(x),w+\tau(z))=2f(x,z)F(y,w),\;\;x,y,z,w\in S,$$
and by \eqref{dalbtype7}, we have
\begin{align}\label{dalbtype9}
F(x+y,z+w)-F(x+\sigma(y),z+\tau(w))=2f(x,z)F(y,w),\;\;x,y,z,w\in S.
\end{align}
Adding \eqref{dalbtype8} and \eqref{dalbtype9}, we find
\begin{align}\label{dalbtype10}
F(x+y,z+w)=F(x,z)f(y,w)+F(y,w)f(x,z),\;\;x,y,z,w\in S.
\end{align}
i.e.,
$$F(u+v)=F(u)f(v)+F(v)f(u),\;\;\text{with}\;\;u=(x,z),\;\;v=(y,w).$$

It follows from this equation that $F$ has the form
\begin{align}\label{dalbtype11}
F(u)=\frac{\chi_1(u)+\chi_2(u)}{2},\;\;\text{with}\;\;u=(x,z),\;\;x,z\in S,
\end{align}
where the functions $\chi_1,\chi_2 : S^2 \to K$  are multiplicative (see \cite{acz,chung,vin}).

Replacing $(x,z)$ by $(x+t,z+s)$ in \eqref{dalbtype10}, we obtain 
\begin{align}\label{dalbtype12}
F(x+t+y,z+s+w)&=F(x+t,z+s)f(y,w)+F(y,w)f(x+t,z+s)\nonumber\\&
\overset{\text{by} \eqref{dalbtype10} }{=}[F(x,z)f(t,s)+F(t,s)f(x,z)]f(y,w)\nonumber\\&\quad+F(y,w)f(x+t,z+s)
\end{align}
for all $x,y,z,w,s,t\in S.$ Replacing $(y,w)$ by $(t+y,s+w)$ in \eqref{dalbtype10}, we get 
\begin{align}\label{dalbtype13}
F(x+t+y,z+s+w)&=F(x,z)f(t+y,s+w)+F(t+y,s+w)f(x,z)\nonumber\\&
\overset{\text{by} \eqref{dalbtype10} }{=}F(x,z)f(t+y,s+w)\nonumber\\&\quad+[F(y,w)f(t,s)+F(t,s)f(y,w)]f(x,z)
\end{align}
for all $x,y,z,w,s,t\in S.$ It follows from \eqref{dalbtype12} and \eqref{dalbtype13} that
$$F(y,w)[f(x+t,z+s)-f(x,z)f(t,s)]=F(x,z)[f(t+y,s+w)-f(t,s)f(y,w)]$$
for all $x,y,z,w,s,t\in S.$

Setting $(y,w)=(x_0,z_0)$ and $h(t,s)=F(x_0,z_0)^{-1}[f(x_0+t,z_0+s)-f(x_0,z_0)f(t,s)]$. We have
\begin{align}\label{dalbtype14}
f(x+t,z+s)-f(t,s)f(x,z)=F(x,z)h(t,s),\;\;x,z,t,s\in S
\end{align}
and
$$ f(x+t,z+s)-f(t,s)f(x,z)=F(t,s)h(x,z),\;\;x,z,t,s\in S.$$
So, $F(x,z)h(t,s)=F(t,s)h(x,z),\;\; x,z,t,s\in S.$  Setting $x=x_0$ and $z=z_0$, we have
$$h(t,s)=F(x_0,z_0)^{-1}F(t,s)h(x_0,z_0).$$
Since $K$ is quadratically closed there exists an $\alpha\in K$ such that $$\alpha^2=F(x_0,z_0)^{-1}h(x_0,z_0).$$ 
Hence the equation \eqref{dalbtype14}, with $(y,w)$ in place of $(t,s)$, becomes
\begin{align}\label{dalbtype15}
f(x+y,z+w)=f(y,w)f(x,z)+\alpha^2 F(x,z)F(y,w),\;\;x,z,y,w\in S.
\end{align}
Multiplying \eqref{dalbtype10} by $\alpha$ and adding to and subtracting it from \eqref{dalbtype15}, we obtain respectively
$$f(x+y,z+w)+\alpha F(x+y,z+w)=[f(y,w)+\alpha F(y,w)][f(x,z)+\alpha F(x,z)]$$
and
$$f(x+y,z+w)-\alpha F(x+y,z+w)=[f(y,w)-\alpha F(y,w)][f(x,z)-\alpha F(x,z)].$$
So the functions $\chi_1:=f+\alpha F$ and $\chi_2:=f-\alpha F$ are multiplicative, and by addition
we get \eqref{dalbtype11}.

Now in view of \eqref{dalbtype3} and \eqref{dalbtype7} we have
\begin{align*}
\chi_1(\sigma(x),\tau(z))&=f(\sigma(x),\tau(z))+\alpha F(\sigma(x),\tau(z))\\&
f(x,z)-\alpha F(x,z)=\chi_2(x,z),\;\; x,z\in S.
\end{align*}
The converse implication follows by simple calculations. This completes the proof of the theorem.
\end{proof}
\begin{rem}
the commutativity of $S$ can be replaced by the conditions
\begin{align}\label{c1}
f(x+t+y,z+s+w)=f(x+y+t,z+w+s),\;\;x,y,z,w,s,t\in S
\end{align}
and \begin{align}\label{c2}
f(x+y,z+w)=f(y+x,w+z),\;\;x,y,z,w\in S.
\end{align}
If $S$ has a neutral element, the second of these conditions is implied by the first
one and can be omitted.
\end{rem}

\begin{cor}[\rm{\cite[Theorem 1]{sino}}]
Let $(S,+)$ be a commutative semigroup,  $K$  a field of
characteristic different from $2$, and let $\sigma: S\to S$ be involution. Then, the general solution $g : S \to K$ of the  d'Alembert  equation:
$$g(x+y)+g(x+\sigma(y)) =2g(x)g(y),\quad x,y\in S,$$
is
$$g(x)=\frac{m(x)+m(\sigma(x))}{2},\quad x\in S,$$
where $m : S \to K$ is an arbitrary multiplicative function (i.e., $m(x+y)=m(x)m(y)$ for $x,y\in S$).
\end{cor}
\begin{proof}  In Theorem \ref {t1}, it suffice to take $z=x,\; y=w,\; \tau=\sigma$ and $g(x):=f(x,x)$ for all $x\in S$.

\end{proof}


\section{General solution of Eq. \eqref{bijns}}\label{jnsen}
\begin{thm}\label{t2}
Let $(S,+)$ be a commutative semigroup,  $G$ a 2-cancellative
abelian group, and let $\sigma,\tau: S\to S$ be involutions. Then, the general solution $f : S^2 \to G$ of the  Jensen type equation:
\begin{align}\label{jnsn}
f(x+y,z+w)+f(x+\sigma(y),z+\tau(w)) =2f(x,z),\quad x,y,z,w\in S,
\end{align}
is
\begin{align}\label{jnsn1}
f(x,y)=a(x,y)+c,\quad\quad x,y\in S,
\end{align}
where $c\in G$ is an arbitrary constant and $a : S^2 \to G$ is an arbitrary additive
function (i.e., $a(x+y,z+w)=a(x,z)+a(y,w)$ for $x,y,z,w\in S$) with $a(\sigma(x),\tau(y))=-a(x,y)$ for all $x,y\in S$.
\end{thm}


\begin{proof}
Replacing $(y,w)$ by $(y+\sigma(y),w+\tau(w))$ in \eqref{jnsn}, we obtain
\begin{align}\label{izdin0}
f(x+y+\sigma(y),z+w+\tau(w))=f(x,z),\quad x,y,z,w\in S.
\end{align}
Putting $(x,z)=(x+u,z+v)$ in \eqref{jnsn}, we get 
\begin{align}\label{izdin1}
f(x+u+y,z+v+w)+f(x+u+\sigma(y),z+v+\tau(w))=2f(x+u,z+v)
\end{align}
for all $x,y,z,u,v,w\in S$. Interchanging $u,y$ and $v,w$ in \eqref{izdin1}, we find
\begin{align}\label{izdin2}
f(x+u+y,z+v+w)+f(x+y+\sigma(u),z+w+\tau(v))=2f(x+y,z+w)
\end{align}
for all $x,y,z,u,v,w\in S$. Adding the last two equations and using \eqref{jnsn}, we have
\begin{align}\label{izdin3}
f(x+u+y,z+v+w)+f(x,z)=f(x+u,z+v)+f(x+y,z+w)
\end{align}
for all $x,y,z,u,v,w\in S$. Replacing $(u,v)$ by $(\sigma(x),\tau(z))$ in \eqref{izdin3}, we get
\begin{align}\label{izdin4}
f(x+\sigma(x)+y,z+\tau(z)+w)+f(x,z)=&f(x+\sigma(x),z+\tau(z))\\&+f(x+y,z+w)\nonumber
\end{align}
for all $x,y,z,,w\in S$. From \eqref{izdin4} and \eqref{izdin0}, we obtain
\begin{align}\label{izdin5}
f(y,w)+f(x,z)=f(x+\sigma(x),z+\tau(z))+f(x+y,z+w)
\end{align}
for all $x,y,z,w\in S$. Interchanging $x,y$ and $z,w$ in \eqref{izdin5}, we have
\begin{align}\label{izdin6}
f(x,z)+f(y,w)=f(y+\sigma(y),w+\tau(w))+f(x+y,z+w)
\end{align}
for all $x,y,z,w\in S$. It follows from \eqref{izdin5}, \eqref{izdin6} that
$$f(x+\sigma(x),z+\tau(z))=f(y+\sigma(y),w+\tau(w)),\quad x,y,z,w\in S.$$
Then $F(x+\sigma(x),z+\tau(z))$ is a constant set, say $c\in G$. So \eqref{izdin6} gives
$$f(x,z)-c+f(y,w)-c=f(x+y,z+w)-c,\;\;x,y,z,w\in S.$$
Putting $a(x,z):=f(x,z)-c$ for $x,y\in S$, so $$a(x+y,z+w)=a(x,z)+a(y,w)$$ for all $x,y,z,w\in S$, this means that $a$ is additive. Putting \eqref{jnsn} into \eqref{jnsn}1 we find $a(\sigma(x),\tau(y))=-a(x,y)$ for all $x,y\in S$ and this completes the proof. 
\end{proof}

\begin{rem}
Here too, the commutativity of S can be replaced by the conditions \eqref{c1} and \eqref{c2}.
\end{rem}

\begin{cor}[\rm{\cite[Theorem 2]{sino}}]
Let $(S,+)$ be a commutative semigroup,   $G$ be a 2-cancellative
abelian group, and let $\sigma: S\to S$ be involution. Then, the general solution $g : S \to G$ of the  Jensen  equation:
\begin{align}\label{dalbtype}
g(x+y)+g(x+\sigma(y)) =2g(x),\quad x,y\in S,
\end{align}
is
\begin{align}\label{dalbtype1}
g(x)=\psi(x)+a,\quad\quad x\in S,
\end{align}
where $a\in G$ is an arbitrary constant and $\psi : S \to G$ is an arbitrary additive
function (i.e., $\psi(x+y)=\psi(x)+\psi(y)$) with $\psi(\sigma(x))=-\psi(x)$ for all $x\in S$.
\end{cor}
\begin{proof}  In Theorem \ref{t2}, it suffice to take $z=x,\; y=w,\; \tau=\sigma$ and $g(x):=f(x,x)$ for all $x\in S$.

\end{proof}
In the proof of the next theorem we shall need the following corollary.

\begin{cor}\label{cor2}
Let $S$, $G$ and $\sigma,\tau$ be as in Theorem \ref{t2}.  Then, the general solution $h : S^2\times S^2 \to G$ of the  Jensen type equation:
\begin{align*}
h((x+y,z+w),(t,s))+h((x+\sigma(y),&z+\tau(w)),(t,s)) =2h((x,z),(t,s)),\\&\quad x,y,z,w,t,s\in S,
\end{align*}
is
\begin{align*}
h((x,y),(t,s))=a((x,y),(t,s))+c(t,s),\quad\quad x,y,t,s\in S,
\end{align*}
where $c:S^2\to G$ is an arbitrary function and $a : S^2\times S^2 \to G$ is an arbitrary 
function additive (i.e., $a((x+y,z+w),(t,s))=a((x,z),(t,s))+a((y,w),(t,s))$ for $x,y,z,w,t,s\in S$) with $a((\sigma(x),\tau(y)),(t,s))=-a((x,y),(t,s))$ for all $x,y,t,s\in S$.
\end{cor}
\begin{proof}
Obvious.
\end{proof}
\section{General solution of Eq. \eqref{biquad}}\label{quad}
\begin{thm}\label{t3}
Let $(S,+)$ be a commutative semigroup,  $H$  an abelian group,
uniquely divisible by 2, and let $\sigma,\tau: S\to S$ be involutions. Then, the general solution $f : S^2 \to H$ of the  quadratic type equation:
\begin{align}\label{qua}
f(x+y,z+w)+f(x+\sigma(y),z+\tau(w)) =2f(x,z)+2f(y,w),\quad x,y,z,w\in S,
\end{align}
is
\begin{align}\label{qua1}
f(u)=B(u,u)+T(u),\;\;u=(x,z),\;\; x,y\in S,
\end{align}
where $T : S^2\to  H$ is an arbitrary additive function (i.e., $T(u+v)=T(u)+T(v)$ where $u=(x,z),\;v=(y,w)$ and $ u+v=(x+y,z+w)$) with $T(\sigma(x),\tau(z))=T(x,z)$ and $B :  S^2\times S^2\to H$ is an arbitrary symmetric biadditive function (i.e., $B(u,v)=B(v,u)$ and $B(u+v,r)=B(u,r)+B(v,r)$ for  $u,v,r\in S^2$) with $B((\sigma(x), \tau(z)),(y,w))=-B((x, z),(y,w))$.
\end{thm}

\begin{proof}
Replacing $(y,w)$ by $(\sigma(y),\tau(w))$ in \eqref{qua}, we see that
\begin{align}\label{qua2}
f(y,w)=f(\sigma(y),\tau(w)),\;\; y,w\in S,
\end{align}
so, 
\begin{align}\label{qua2+}
f(x+\sigma(y),z+\tau(w))=f(\sigma(x)+y,\tau(z)+w),\;\; x,z,y,w\in S.
\end{align}

Also in \eqref{qua}, we replace $(x,z)$ first by $(x +t,z+s)$ and then by $(x +\sigma(t),z+\tau(s))$, and find
\begin{align*}
f(x+t+y,z+s+w)&+f(x+t+\sigma(y),z+s+\tau(w)) \\&=2f(x+t,z+s)+2f(y,w),\quad x,y,z,w,t,s\in S,
\end{align*}
\begin{align*}
f(x+\sigma(t)&+y,z+\tau(s)+w)+f(x+\sigma(t)+\sigma(y),z+\tau(s)+\tau(w)) \\&=2f(x+\sigma(t),z+\tau(s))+2f(y,w),\quad x,y,z,w,t,s\in S.
\end{align*}
By subtraction we obtain
\begin{align}\label{qua3}
h((x+y,z+w),(t,s))+h((x+\sigma(y),z+\tau(w)),(t,s))=2h((x,z),(t,s))
\end{align}
for all $x,z,t,s,y,w\in S,$ where 
\begin{align}\label{qua4}
h((x,z),(t,s)):&=\frac{f(x+t,z+s)-f(x+\sigma(t),z+\tau(s))}{4}
\end{align}
for all $x,z,t,s\in S$. According to Corollary \ref{cor2}, we have $h(u,r)=B(u,r)+c(r)$ where $B : S^2\times S^2 \to H$ is an arbitrary function additive, i.e., $B(u+v,r)=B(u,r)+B(v,r)$ for all $u,v,r\in S^2$ (with $u=(x,z),\;v=(y,w),\; u+v=(x+y,z+w)$ and $r=(t,s)$) and $B((\sigma(x),\tau(y)),r)=-B((x,y),r)$. Hence 
$$h((x,z),r)+h((\sigma(x),\tau(z)),r)=2c(r).$$
On the other hand, from \eqref{qua4} and \eqref{qua2+}, we have $h(u,r)=h(r,u)$ and, in view of \eqref{qua2} and \eqref{qua2+}, $h((\sigma(x),\tau(z)),r)=-h((x,z),r).$ So, $c(r)=0$ and $h(u,r)=b(u,r)$, that is, $b$ is symmetric and biadditive.

Now from \eqref{qua} with $(y,w)=(x,z)$, we find $f(2x,2z)=4f(x,z)-f(x+\sigma(x),z+\tau(z))$ and from \eqref{qua3} with $(t,s)=(x,z)$, we have $h((x,z),(x,z))=\frac{1}{4}[f(2x,2z)-f(x+\sigma(x),z+\tau(z))]$. Hence
$$f(u)=B(u,u)+\frac{1}{2}f(x+\sigma(x),z+\tau(z)).$$

Putting $(x,y)=(x+\sigma(x),y+\sigma(y))$ and $(z,w)=(z+\tau(z),w+\tau(w))$ in \eqref{qua}, we get
\begin{align*}
f(x+\sigma(x)+y+\sigma(y),z+\tau(z)&+w+\tau(w))=f(x+\sigma(x),z+\tau(z))\\&+f(y+\sigma(y),w+\tau(w)),\quad x,y,z,w\in S,
\end{align*}
which means that the function $T(x,z):=\frac{1}{2}f(x+\sigma(x),z+\tau(z))$ is additive on $S^2$. Also 
$T(\sigma(x),\tau(z))=T(x,z)$ for all $x,z\in S$. 

Conversely, it is easy to check that any function $f$ of the form \eqref{qua1} satisfies \eqref{qua}. This completes the proof of the theorem.
\end{proof}

\begin{cor}
Let $(S,+)$ be a commutative semigroup,  $H$  an abelian group,
uniquely divisible by 2, and let $\sigma: S\to S$ be involution. Then, the general solution $g : S \to H$ of the  quadratic  equation:
$$g(x+y)+g(x+\sigma(y)) =2f(x)+2f(y),\quad x,y\in S,$$
is
$$f(x)=b(x,x)+\psi(x),\;\; x\in S,$$
where $\psi : S\to  H$ is an arbitrary additive function  with $\psi(\sigma(x))=\psi(x)$ and $b:  S^2\to H$ is an arbitrary symmetric biadditive function with $b(\sigma(x), y)=-b(x,y)$.
\end{cor}

\begin{proof}  In Theorem \ref{t3}, it suffice to take $z=x,\; y=w,\; \tau=\sigma$ and $g(x):=f(x,x)$ for all $x\in S$.
\end{proof}

We end the paper with some questions and a conclusion.
\begin{qes}
1) As a future work, what are the general solutions of the following equations:
$$
f(x+y+\alpha,z+w+\beta)+f(x+\sigma(y)+\alpha,z+\tau(w)+\beta) =2f(x,z)f(y,w)
$$
$$
f(x+\sigma(y)+\alpha,z+\tau(w)+\beta)-f(x+y+\alpha,z+w+\beta) =2f(x,z)f(y,w)
$$
$$
f(x+y,z+w)+f(x+\sigma(y),z+\tau(w)) =2f(x,z)g(y,w)
$$
$$
f(x+y,z+w)+g(x+\sigma(y),z+\tau(w)) =h(x,z)
$$
$$
f(x+y,z+w)+f(x+\sigma(y),z+\tau(w)) =2f(x,z)+2g(y,w)
$$
$$
f(x+y,z+w)+g(x+\sigma(y),z+\tau(w)) =2f(x,z)+2h(y,w)
$$
for all $\alpha,\beta,x,y,z,w\in S$, where  $S$is a semi-group and $\sigma,\tau: S\to S$ are involutions.

2) We can also study the Hyers-Ulam stability of the previous functional equations.
\end{qes}

\section*{Conclusion} \label{concla}
In this work, we managed to find general solutions to the functional equations \eqref{bialm}, \eqref{bijns} and \eqref{biquad}. From these equations we can derive a new type of functional equation which will be of utmost importance in the future. This could be a potential future work.

\end{document}